\documentclass{amsart}

\makeatletter
 \def\LaTeX{\leavevmode L\raise.42ex
   \hbox{\kern-.3em\size{\sf@size}{0pt}\selectfont A}\kern-.15em\TeX}
\makeatother

\newcommand{\BibTeX}{{\rm B\kern-.05em{\sc
i\kern-.025emb}\kern-.08em\TeX}}

\newtheorem{thm}{Theorem}[section]
\newtheorem{lem}[thm]{Lemma}

\newtheorem{rem}[thm]{Remark}
\theoremstyle{defn}
\newtheorem{defn}{Definition}

\numberwithin{equation}{section}

\begin{document}

\title[Almost Parseval frames on  noncompact symmetric spaces]{ Paley-Wiener-Schwartz nearly Parseval frames and Besov spaces on  noncompact symmetric spaces}

\author{Isaac Z. Pesenson }\footnote{ Department of Mathematics, Temple University,
 Philadelphia,
PA 19122; pesenson@temple.edu. The author was supported in
part by the National Geospatial-Intelligence Agency University
Research Initiative (NURI), grant HM1582-08-1-0019. }


\begin{abstract}
Let $X$ be a symmetric space of the noncompact type. The goal of the paper is to construct in the space $L_{2}(X)$ nearly Parseval frames consisting of functions which simultaneously  belong to Paley-Wiener spaces and to Schwartz space on $X$. We call them Paley-Wiener-Schwartz frames in $L_{2}(X)$. These frames are used to characterize a family of Besov spaces on $X$. As a part of our construction we develop on $X$ the so-called average Shannon-type sampling.

\end{abstract}

\maketitle

\section{Introduction}
Wavelet systems  and frames which build up of bandlimited functions with a strong localization on the space  became very popular in theoretical and applied analysis on Euclidean spaces. Frames with similar properties were  recently constructed in non-Euclidean settings in $L_{2}$-spaces on spheres \cite{NPW}, compact Riemannian manifolds \cite{gm3}-\cite{gpes},  certain metric-measure spaces \cite{CKP}. In \cite{gpes} Parseval bandlimited and localized frames were developed on compact homogeneous manifolds. Bandlimited ($\equiv$ Paley-Wiener) frames  in Paley-Wiener spaces on noncompact manifolds of bounded geometry (in particular on noncompact symmetric spaces) were constructed in \cite{Pes00}-\cite{Pes11}, \cite{FP04}, \cite{FP05}, \cite{CGSM}, \cite{EEKK1}, \cite{EEKK2}. In \cite{Pes98a}, \cite{Pes98b} bandlimited  frames in Paley-Wiener spaces (in subelliptic framework) were constructed on stratified Lie groups. Bandlimited and localized nearly Parseval frames on domains in $\mathbb{R}^{n}$ with smooth boundaries can be found in \cite{Pes12}.

The goal of our development is to construct Paley-Wiener localized (Schwartz) frames in $L_{2}$ spaces on symmetric manifolds of the noncompact type. It should be noted that such manifolds do not satisfy requirements of the paper \cite{CKP}. 

A Riemannian symmetric space of the noncompact type is a
Riemannian manifold $X$ of the form $X=G/K$ where $G$ is a
connected semisimple Lie group with finite center and $K$ is a
maximal compact subgroup of $G$. The most known examples of such spaces are the real, complex, and quaternionic hyperbolic spaces. 

Let $f$ be a function in the space  $L_{2}(X, d\mu(x))=L_{2}(X)$, where
$X$ is a symmetric space of noncompact type and $dx$ is an
invariant measure. The notations $\mathcal{F}f=\widehat{f}$ will be used for the
Helgason-Fourier transform of $f$. The Helgason-Fourier transform
$\widehat{f}$ can be treated as a function on
$\mathbb{R}^{n}\times \mathcal{B}$ where $\mathcal{B}$ is a
certain compact homogeneous manifold and $n$ is the rank of $X$.
Moreover, $\widehat{f}$ belongs to the space
$
 L_{2}\left(\mathbb{R}^{n}\times
\mathcal{B};|c(\lambda)|^{-2}d\lambda db \right),
$
where $c(\lambda)$ is the Harish-Chandra's function, $d\lambda$ is
the Euclidean measure and $db$ is the normalized invariant measure
on $\mathcal{B}$. The notation
$\Pi_{[\omega_{1},\>\omega_{2}]}\subset \mathbb{R}^{n}\times \mathcal{B},\>\>0<\omega_{1}<\omega_{2},$ will be
used for the set of all points $(\lambda, b)\in
\mathbb{R}^{n}\times \mathcal{B}$ for which
$\omega_{1}\leq \sqrt{\left<\lambda,\lambda\right>}\leq \omega_{2},$ where $\left<\cdot,
\cdot\right>$ is the Killing form. In particular the notation 
$\Pi_{\omega}\subset \mathbb{R}^{n}\times \mathcal{B}$ will be
used for $\Pi_{[0,\>\omega]}.$

The  Paley-Wiener space $PW_{[\omega_{1}, \>\omega_{2}]}(X), 0<\omega_{1}<\omega_{2},$ is defined as
the set of all functions in  $L_{2}(X)$ whose Helgason-Fourier
transform has support in $\Pi_{[\omega_{1},\>\omega_{2}] }$ and belongs to the space
$
\Lambda_{\omega}=L_{2}\left(\Pi_{[\omega_{1},\>\omega_{2}] };|c(\lambda)|^{-2}d\lambda
db\right).
$
In particular, $PW_{ \omega}(X)$ will be used for $PW_{[0,\> \omega]}(X)$.

Using  the $K$-invariant distance on $X$  to the "origin" (see formula (\ref{X-norm}) below) one can introduce a  notion of an $L_{2}$-Schwartz space $S^{2}(X)$ (see Definition \ref{Ss} below) which was considered by M. Eguchi \cite{Egu}. A theorem of M. Eguchi \cite{Egu} states that  if a function $f$ is in $C_{0}^{\infty}(\mathbb{R}^{n}\times \mathcal{B})$ and satisfies certain symmetry conditions (see (\ref{symmetry} below))  then its inverse Helgason-Fourier transform is a function in $S^{2}(X)$.
 A refinement of this result was  given by N. B. Andersen in \cite{A} (see  Theorems \ref{E} and \ref{A} below).

Although some facts that we discuss in this development already appeared in our previous papers the main result about 
 existence of Paley-Wiener-Schwartz nearly Parseval frames in $L_{2}(X)$ is completely new. Here is a formulation of our main Theorem. 
\begin{thm}\label{FrTh}
 Suppose that  $X$ is a  Riemannian symmetric spaces of the noncompact type.
  For every $0<\delta<1$ there exists a countable family of functions $\{\Theta_{j , \gamma}\}$ such that
 
 \begin{enumerate}
 
 \item Every function $\Theta_{j, \gamma}$ is bandlimited to $[2^{j-1}, 2^{j+1}]$ in the sense that  $\Theta_{j, \gamma}\in PW_{[2^{j-1}, 2^{j+1}]}(X)$.
 
 \item Every function $\Theta_{j, \gamma}$ belongs to $S^{2}(X)$. 
 
 \item $\{\Theta_{j , \gamma}\}$  is a frame in $L_{2}(X)$ with constants $1-\delta$ and $1+\delta$, i.e.
 $$
 (1-\delta)\|f\|^{2}\leq \sum_{j\in \mathbb{N}}\sum_{\gamma}\left|\left<f, \Theta_{j, \gamma}\right>\right|^{2}\leq (1+\delta)\|f\|^{2},\>\>\>f\in L_{2}(X).
 $$

 \end{enumerate}
 
 \end{thm}

In section \ref{HarmAnal} we summarize basic facts about  harmonic analysis on Riemannian symmetric spaces of the noncompact type. In subsection \ref{Riem} we prove a  covering Lemma for Riemannian  manifolds  of bounded geometry whose Ricci curvature is bounded from below. In section \ref{PW-100} we introduce Paley-Wiener spaces $PW_{\omega}(X),\>\>\>\omega>0$. In section \ref{AvS} we develop average sampling and almost Parseval frames in Paley-Wiener spaces on Riemannian manifolds. The main result is obtained in section \ref{FFFF}  where we construct nearly Parseval Paley-Wiener-Schwartz  frames in $L_{2}(X)$. In the last section we characterize Besov spaces $\mathbf{B}_{2, q}^{\alpha}(X)$ in terms of coefficients with respect to constructed frames. The entire scale of Besov spaces will be considered in a separate paper.

\section{ Harmonic analysis on Riemannian symmetric spaces of the noncompact type}\label{HarmAnal}

\subsection{ Riemannian symmetric spaces of the noncompact type}

A Riemannian symmetric space of the noncompact type is a
Riemannian manifold $X$ of the form $X=G/K$ where $G$ is a
connected semisimple Lie group with finite center and $K$ is a
maximal compact subgroup of $G$. The Lie algebras of the groups
$G$ and $K$ will be denoted respectively as $\textbf{g}$ and
$\textbf{k}$. The group $G$ acts on $X$ by left translations. If
$e$ is the identity in $G$  then the base point  $eK$ is denoted
by $0$. Every such $G$ admits Iwasawa decomposition $G=NAK$, where
the nilpotent Lie group $N$ and the abelian group $A$ have Lie
algebras $\textbf{n}$ and $\textbf{a}$ respectively. Correspondingly $\textbf{g}=\textbf{n} \bigoplus\textbf{a} \bigoplus\textbf{k}.$ The dimension
of $\textbf{a}$ is known as the rank of $X$.
 The letter $M$ is usually used to denote the centralizer of $A$ in
$K$ and the letter $\mathcal{B}$ is commonly used for the
homogeneous space $K/M$.

The Killing form on $\textbf{g}$ induces an Ad$K$-invariant inner product on $\textbf{n} \bigoplus\textbf{a}$ which generates a $G$-invariant Riemannian metric on $X$. With this metric $X=G/K$ becomes a Riemannian globally symmetric space of the noncompact type.

Let $\textbf{a}^{*}$ be the real dual of $\textbf{a}$ and $W$ be
the Weyl's group. Let $\Sigma$  be the set of
restricted roots, and $\Sigma^{+}$ will be the set of all positive
roots. The notation $\textbf{a}^{+}$ has the following meaning
$$
\textbf{a}^{+}=\{H\in \textbf{a}|\alpha(H)>0, \alpha\in
\Sigma^{+}\}
$$
 and is known as positive Weyl's chamber. Let $\rho\in
 \textbf{a}^{*}$ is defined in a way that $2\rho$ is the sum of
 all  positive restricted roots. The Killing form $<,>$ on $\textbf{a}$
 defines a
 metric on $\textbf{a}$. By duality it defines an inner  product on
 $\textbf{a}^{*}$.

 We denote by  $\textbf{a}^{*}_{+}$ the set of
 $\lambda\in \textbf{a}^{*}$, whose dual belongs to
 $\textbf{a}^{+}$.
According to Iwasawa decomposition for every $g\in G$ there exist
a unique $A(g), \>H(g)\in \textbf{a}$ such that
$$g=n \exp A(g) k=k\exp H(g)n,\>\> k\in K, \>\>n\in N,\>\>A(g)=-H(g^{-1}),
$$
 where $\exp :\textbf{a}\rightarrow A$ is the exponential map of
 the
 Lie algebra $\textbf{a}$ to Lie group $A$. On the direct product
 $X\times \mathcal{B}$ we introduce function with values in $\textbf{a}$
 using the formula
 \begin{equation}
 A(x,b)=A(u^{-1}g)
 \end{equation}
 where $x=gK, g\in G, b=uX, u\in K$.

According to Cartan decomposition every element $g$ of $G$ has representation $g=k_{1}\exp(H)k_{2}$, where $H$ belongs to the closure of $\exp \textbf{a}^{+}$. The norm of an $g$ in $G$ is introduced as
\begin{equation}\label{X-norm}
|g|=|k_{1}\exp(H)k_{2}|=\|H\|.
\end{equation}
It is the $K$-invariant geodesic distance on $X$  of $gK$ to $eK$.

\subsection{ A covering Lemma for Riemannian  manifolds  of bounded geometry whose Ricci curvature is bounded from below.}\label{Riem}

Let $X$, dim$X=d$, be a connected $C^{\infty}-$smooth Riemannian
manifold with a $(2,0)$ metric tensor $g$ that defines an inner
product on every tangent space $T_{x}(X), x\in X$. The
corresponding Riemannian distance $d$ on $X$ and
the Riemannian measure $d\mu(x)$ on $X$ are given by
$$
d(x,y)=inf\int_{a}^{b}
\sqrt{g\left(\frac{d\alpha}{dt},\frac{d\alpha}{dt}\right)}dt,\>\>\>\>d\mu(x)=\sqrt{|det(g_{ij})|}dx,
$$
where the infimum is taken over all $C^{1}-$curves
$\alpha:[a,b]\rightarrow X, \alpha(a)=x, \alpha(b)=y,$
the $\{g_{ij}\}$ are the components of the tensor $g$ in a
local coordinate system and  $dx$ is the Lebesgue's measure in
$R^{d}$.  Let $ exp_ {x} $ : $T_{x}(X)\rightarrow X,$ be the exponential
geodesic map i. e. $exp_{x}(u)=\gamma (1), u\in T_{x}(X),$ where
$\gamma (t)$ is the geodesic starting at $x$ with the initial
vector $u$ : $\gamma (0)=x , \frac{d\gamma (0)}{dt}=u.$ 
We denote by $\operatorname{inj}$ the largest real number $r$ such that $exp_x$ is a diffeomorphism of a suitable open neighborhood of $0$ in $T_x X$ onto $B(x,r)$, for all sufficiently small $ r$ and $x \in X$. Thus for every choice of an orthonormal basis (with respect to the inner
 product defined by $g$) of $T_{x}(X)$ the exponential map
  $exp$ defines a
 coordinate system on $B(x, r)$ which is called {\it geodesic}.
The volume of the ball $B(x, r)$ will be denoted
by $|B(x, r)|.$
Throughout the paper we will consider only geodesic coordinate
 systems.

A Riemannian symmetric space $X$ equipped with an invariant metric has bounded geometry which mean that 

(a) \textit {$X$ is complete and connected;}

(b)  \textit {the injectivity radius $\operatorname{inj}(X)$ is positive;}

(c)  \textit {for any $r\leq \operatorname{inj}(X)$, and for every two canonical
coordinate systems \\ $\vartheta_{x}: T_{x}(X)\rightarrow B(x, r),
\vartheta_{y}:T_{y}(X)\rightarrow B(x, r),$ 
the following inequalities holds true:}
$$
\sup _{x\in B(x, r)\cap B(y, r)}\sup_{|\alpha|\leq k}|\partial
^{|\alpha|}\vartheta_{x}^{-1}\vartheta_{y}|\leq C( r, k).
$$
The Ricci curvature $Ric$ of $X$ is bounded from below, i.e.
\begin{equation}\label{Ric}
Ric\geq-kg, \>\>\>k\geq 0
\end{equation}
According to  the Bishop-Gromov Comparison Theorem this fact implies
the so-called {\it local doubling property}:  for any 
 $0 <\sigma<\lambda< r< \operatorname{inj}(X)$:
\begin{equation}\label{BG}
|B(x,\lambda)|\leq\left(\lambda/\sigma\right)^{d}e^{(k
r(d-1))^{1/2}} |B(x,\sigma)|,\>d=dim\>X.
\end{equation}

We will need the following lemma which was proved in \cite{Pes04b}.
\begin{lem}\label{cover}
If $X$ is a Riemannian manifold  of bounded geometry and its  Ricci curvature is bounded from below then there exists a natural $N_{X}$ such that for any
$0<r<\operatorname{inj}\>(X)$ there exists a set of points 
$X_{r}=\{x_{i}\}$ with the following properties
\begin{enumerate}
\item  \textit {the balls $B(x_{i},  r/4)$ are disjoint,}
\item  \textit {the balls $B(x_{i}, r/2)$ form a cover of $X$,}
\item  \textit {the height of the cover by the balls $B(x_{i},r)$ is not
greater than $N_{X}.$}
\end{enumerate}
\end{lem}
\begin{proof}
 Assumptions (a)-(c)  imply that there exist constants $a, b>0$ such that
\begin{equation}\label{bb}
a\leq\frac{|B(x, r)|}{|B(y, r)|} \leq b, x,y\in X,
r< \operatorname{inj}(X),
\end{equation}
 where $\operatorname{inj}(X)$ is the injectivity radius
 of the manifold.
Let us choose a family of disjoint balls $B(x_{i}, r/4)$ such
that there is no ball $B(x, r/4), x\in X,$ which has empty
intersections with all balls from our family. Then the family
$B(x_{i},r/2)$ is a cover of $X$. Every ball from the family
$\{B(x_{i}, r)\}$ having non-empty intersection with a
particular ball $B(x_{j}, r)$ is contained in the ball
$B(x_{j}, 3r)$. Since any two balls from the family
$\{B(x_{i},r/4)\}$ are disjoint, the inequalities (\ref{bb}) and (\ref{BG}) give the following
estimate for the multiplicity $N$ of the covering
$\{B(x_{i},r)\}$:
\begin{equation}
N\leq\frac{\sup_{y\in X}|B(y,3r)|}{\inf_{x\in X}|B(x,r/4)|} \leq C(X)b
12^{d}=N_{X},\>\>\>d=dim\>X.
\end{equation}
Thus the lemma is proved.
\end{proof}

\begin{defn}
Every set of point $X_{r}=\{x_{i}\}$ that satisfies conditions of Lemma \ref{cover} will be called a $r$-lattice.
\end{defn}

 To construct Sobolev spaces $H^{k}(X), \>\>
k\in \mathbb{N},$ we fix a $\lambda$-lattice $X_{\lambda}=\{y_{\nu}\}, \>\>0< \lambda< \operatorname{inj} (X)$ and introduce a partition of unity  ${\varphi_{\nu}}$ that is 
{\it subordinate to the family} $\{B(y_{\nu},\lambda/2)\}$ and has the
following properties:
\begin{enumerate}
\item $\varphi_{\nu}\in C_{0}^{\infty} B(y_{\nu},\lambda/2 ),$

\item $\sup_{x}\sup_{|\alpha|\leq k}|\varphi_{\nu}^{(\alpha
)}(x)|\leq C(k), $ where $ C(k) $ is independent on $\nu$ for
every $k $ in geodesic  coordinates.
\end{enumerate}

The exponential map $exp_{y_{\nu}}:T_{y_{\nu}}M\rightarrow M$ is a diffeomorphism of a ball $B_{T_{y_{\nu}}}(0, r)\subset T_{y_{\nu}}M$ with center $0$ and of radius $r$ on a ball $B(y_{\nu}, r)$ in Riemannian metric on $M$ (assuming that $r>0$ is sufficiently small). If $M$ is homogeneous  and a metric is invariant then  every ball $B(y_{\nu}, r)$ is a translation on a single ball. In this case there exist two constants $c_{1},\>C_{1}$ such that for any ball $B(x, \>r)$ with $x\in B(y_{\nu}, \>r)$ and $\rho<r$ one has 
\begin{equation}\label{ball-condtion}
B_{T_{y_{\nu}}}(exp_{y_{\nu}}^{-1}(x), c_{1}r)
\subset exp_{y_{\nu}}^{-1}\left(B(x, \>r)\cap B(y_{\nu}, r)\right)\subset B_{T_{y_{\nu}}}(exp_{y_{\nu}}^{-1}(x), C_{1}r).
\end{equation}
Note that for a Riemannian measure $d\mu(x)$ and a locally integrable function $F$ on $U\subset M$ the integral of $F$ over $U$ is defined as follows
\begin{equation}\label{integral}
\int_{U}F(x)d\mu(x)=\int _{exp^{-1}_{y_{\nu}}(U)}F \circ exp_{y_{\nu}}(x_{1},...,x_{d})\sqrt{|det(g_{ij})|}\>dx_{1}...dx_{d},
\end{equation}
where $g_{ij}=g(\partial_{i}, \>\partial_{j})$, and $g$ is the Riemann inner product in tangent space. By choosing a basis $\partial_{1},...,\partial_{d}$,  which is orthonormal with respect to $g$ we obtain $|det(g_{ij})|=1$.

We introduce the {\it Sobolev space} $H^{k}(X), k\in \mathbb{N},$ as the
completion of $C_{0}^{\infty}(X)$ with respect to the norm
\begin{equation}\label{Sob}
\|f\|_{H^{k}(X)}=\left(\sum_{\nu}\|\varphi_{\nu}f\|^{2}
_{H^{k}(B(y_{\nu}, \lambda/2))}\right) ^{1/2},
\end{equation}
where 
$$\|\varphi_{\nu}f\|^{2}
_{H^{k}(B(y_{\nu}, \lambda/2))}=\sum_{1\leq |\alpha\leq k}\|\partial ^{|\alpha|}\varphi_{\nu}f\|^{2}_{L_{2}(B(y_{\nu}, \lambda/2))}
$$ 
and all partial derivatives are  computed in a fixed canonical coordinate system $\exp^{-1}_{y_{\nu}}    $ on $B(y_{\nu}, \lambda/2)$. 
\begin{rem}
A geodesic  coordinate system $\exp^{-1}_{y}$ depends on the choice of a  basis in the tangent space $T_{y},\>\>y\in X$. We assume that such basis is fixed and orthonormal for every $y=y_{\nu}\in X_{r}=\{y_{\nu}\}$.
\end{rem}

The Laplace-Beltrami which is given in a local coordinate system  by the formula
$$
\Delta
f=\sum_{m,k}\frac{1}{\sqrt{det(g_{ij})}}\partial_{m}\left(\sqrt{det(g_{ij})}
g^{mk}\partial_{k}f\right)
$$
where $g_{ij}$ are components of the metric tensor,$\>\>det(g_{ij})$ is the determinant of the matrix $(g_{ij})$, $\>\>g^{mk}$ components of the matrix inverse to $(g_{ij})$.
 It is known that the operator $(-\Delta)$ is a self-adjoint positive
definite operator in the corresponding space $L_{2}(X,d\mu(x)),$ where
$d\mu(x)$ is the $G$-invariant measure. The regularity Theorem for the
Laplace-Beltrami operator $\Delta$ states that domains of the
powers
 $(-\Delta)^{s/2}$ coincide with the Sobolev spaces
$H^{s}(X)$ and  the norm (\ref{Sob}) is equivalent to the
graph norm $\|f\|+\|(-\Delta)^{s/2}f\|$ (see \cite{T3}, Sec.
7.4.5.) Moreover, since the operator $\Delta$ is invertible in
$L_{2}(X)$ the Sobolev norm is also equivalent to the norm
$\|(-\Delta)^{s/2}f\|.$

\subsection{Helgason-Fourier transform on  Riemannian symmetric spaces of the noncompact type}
For every $f\in C_{0}^{\infty}(X)$ the Helgason-Fourier transform
is defined by the formula
\begin{equation}\label{H-F}
\widehat{f}(\lambda,b)=\int_{X}f(x)e^{(-i\lambda+\rho)(A(x,b))}dx,
\end{equation}
where $ \lambda\in \textbf{a}^{*}, b\in \mathcal{B}=K/X, $ and
$dx$ is a $G$-invariant measure on $X$. This integral can also be
expressed as an integral over the group $G$. Namely, if $b=uX,u\in
K$, then
\begin{equation}
\widehat{f}(\lambda,b)=\int_{G}f(gK)e^{(-i\lambda+\rho)(A(u^{-1}g))}dg.
\end{equation}
The invariant measure on $X$ can be normalized so that the
following inversion formula holds for $f\in C_{0}^{\infty}(X)$
$$
f(x)=w^{-1}\int_{\textbf{a}^{*}\times
\mathcal{B}}\widehat{f}(\lambda,b)e^{(i\lambda+\rho)(A(x,b))}|c(\lambda)|^{-2}d\lambda
db,
$$
where $w$ is the order of the Weyl's group and $c(\lambda)$ is the
Harish-Chandra's function, $d\lambda$ is the Euclidean measure on
$\textbf{a}^{*}$ and $db$ is the normalized $K$-invariant measure
on $\mathcal{B}$. This transform can be extended to an isomorphism
between the spaces $L_{2}(X,d\mu(x))$ and
$L_{2}(\textbf{a}^{*}_{+}\times \mathcal{B},
|c(\lambda)|^{-2}d\lambda db)$ and the Parseval's  formula holds
true
$$
\int_{X}f_{1}(x)\overline{f_{2}(x)}d\mu(x)=\int_{\textbf{a}^{*}_{+}\times \mathcal{B}}\widehat{f_{1}
(\lambda,b)}\widehat{f_{2}(\lambda,b)}c(\lambda)|^{-2}d\lambda db
$$
which implies the Plancherel's formula
$$
\|f\|=\left( \int_{\textbf{a}^{*}_{+}\times \mathcal{B}}|\widehat{f}
(\lambda,b)|^{2}|c(\lambda)|^{-2}d\lambda db\right)^{1/2}.
$$

Let $\Delta$ be the Laplace-Beltrami
operator of the $G$-invariant Riemannian structure on $X$.
It is known that the following formula holds
\begin{equation}
\widehat{\Delta
f}(\lambda,b)=-\left(\|\lambda\|^{2}+\|\rho\|^{2}\right)\widehat{f}(\lambda,b),
f\in C_{0}^{\infty}(X),
\end{equation}
where
$\|\lambda\|^{2}=\left<\lambda,\lambda\right>,\|\rho\|^{2}=\left<\rho,\rho\right>,
\left<\cdot,\cdot\right>$ is the Killing form.

\subsection{A Paley-Wiener Theorem on $X$}

A function $\phi(\lambda, b)$ in $C^{\infty}(\mathbf{a}^{*}_{\mathbb{C}}\times \mathcal{B})$, holomorphic in $\lambda$, is called a holomorphic
function of uniform exponential type $\sigma$, if there exists a constant
$\sigma \geq  0$, such that, for each $N\in \mathbb{N} $ one has
$$
\sup_{(\lambda, b) \in \mathbf{a}^{*}_{\mathbb{C}}\times \mathcal{B}}e^{-\sigma|\Im \lambda|}(1+|\lambda|)^{N}|\phi(\lambda, b)|<\infty.
$$
The space of all holomorphic functions of uniform exponential type $\sigma$  will be
denoted $\mathcal{H}_{\sigma}(\textbf{a}^{*}_{\mathbb{C}}\times \mathcal{B})$ and   
$$
\mathcal{H}(\textbf{a}^{*}_{\mathbb{C}}\times \mathcal{B})=\bigcup _{\sigma>0}\mathcal{H}_{\sigma}(\textbf{a}^{*}_{\mathbb{C}}\times \mathcal{B}).
$$
One also need the space $
\mathcal{H}(\textbf{a}^{*}_{\mathbb{C}}\times \mathcal{B})^{W}$ of all functions $\phi\in  
\mathcal{H}(\textbf{a}^{*}_{\mathbb{C}}\times \mathcal{B})$ that satisfy the following property
\begin{equation}\label{symmetry}
\int_{\mathcal{B}}e^{(iw\lambda+\rho)(A(x,b))}\phi(w\lambda, b)db=\int_{\mathcal{B}}e^{(i\lambda+\rho)(A(x,b))}\phi(\lambda, b)db,
\end{equation}
for all $w\in W$ and all $\lambda\in \textbf{a}^{*}_{\mathbb{C}}, \>x\in X$.
 
The following  analog of the Paley-Wiener Theorem is known.
\begin{thm}
The  Helgason-Fourier transform (\ref{H-F}) is a bijection of  $C_{0}^{\infty}(X)$ onto the space $
\mathcal{H}(\textbf{a}^{*}_{\mathbb{C}}\times \mathcal{B})^{W}$ and the inverse of this bijection can be expressed as 
\begin{equation}\label{IHF}
f(x)=\int_{\textbf{a}^{*}_{+}\times
\mathcal{B}}\hat{f}(\lambda,b)e^{(i\lambda+\rho)(A(x,b))}|c(\lambda)|^{-2}d\lambda
db.
\end{equation}
In particular, $\widehat{f}$ belongs to the space $
\mathcal{H}_{\sigma}(\textbf{a}^{*}_{\mathbb{C}}\times \mathcal{B})^{W}$ if and only if the support of $f$ is in the ball $B_{\sigma}$. Here $B_{\sigma}$ is the ball in invariant metric on $X$ whose radius is $\sigma$ and center is $eK$.
\end{thm}

\section{Paley-Wiener spaces $PW_{\omega}(X)$}\label{PW-100}

\begin{defn}
We will say that $f\in L_{2}(X,d\mu(x))$ belongs to the class
$PW_{\omega}(X)$ if its Helgason-Fourier transform $\widehat{f}\in L_{2}(\textbf{a}^{*}_{+}\times
\mathcal{B})$ has compact
support in the sense that $\widehat{f}(\lambda,b)=0$ a. e. for
$\|\lambda\|>\omega$. Such functions will  be also called
$\omega$-band limited.
\end{defn}

Using the spectral resolution of identity $P_{\lambda}$ we define
the unitary group of operators by the formula
$$
e^{it\Delta}f=\int_{0}^{\infty}e^{it\tau}dP_{\tau}f,\>\>\> f\in
L_{2}(X), \>\>\>t\in \mathbb{R}.
$$
Let us introduce the operator
\begin{equation}
\textbf{R}_{\Delta}^{\sigma}f=\frac{\sigma}{\pi^{2}}\sum_{k\in\mathbb{Z}}\frac{(-1)^{k-1}}{(k-1/2)^{2}}
e^{i\left(\frac{\pi}{\sigma}(k-1/2)\right)\Delta}f,\>\>\> f\in L_{2}(X),\>\>\>
\sigma>0.\label{Riesz1}
\end{equation}
 Since  $\left\|e^{it\Delta}f\right\|=\|f\| $ and
\begin{equation}
\frac{\sigma}{\pi^{2}}\sum_{k\in\mathbb{Z}}\frac{1}{(k-1/2)^{2}}=\sigma,\label{id}
 \end{equation}
 the series in (\ref{Riesz1}) is convergent and it shows that
 $\textbf{R}_{\Delta}^{\sigma}$ is a bounded operator in $L_{2}(X)$
 with the norm $\sigma$:
\begin{equation}
 \|\textbf{R}_{\Delta}^{\sigma}f\|\leq \sigma\|f\|, f\in
L_{2}(X).\label{Riesznorm}
 \end{equation}

 The next theorem contains generalizations of several results
 from the classical harmonic analysis (in particular  the Paley-Wiener theorem)
  and it  follows essentially
from our more general results in \cite{Pes00}-
\cite{Pes12}(see also \cite{A}, \cite{Pa}).
\begin{thm} Let $f\in L_2(X)$. Then the following
statements are equivalent:
\begin{enumerate}
\item $f\in PW_{\omega}(X)$;
 \item $f\in
C^{\infty}(X)=\bigcap_{k=1}^{\infty}H^{k}(X), $ and for all $s\in
\mathbb{R}_{+}$ the following Bernstein inequality holds:
\begin{equation}
\|\Delta^{s}f\|\leq( \omega^{2}+\|\rho\|^{2})^{s}\|f\|;\label{B}
\end{equation}

\item $f\in C^{\infty}(X)$ and the
following Riesz interpolation formula holds
\begin{equation}
\left(i \Delta\right)^{n}f=\left(\textbf{R}_{\Delta}^{
\omega^{2}+\|\rho\|^{2}}\right)^{n}f, n\in \mathbb{N};
\label{Rieszn}
\end{equation}

\item  For every $g\in
L_{2}(X)$ the function $t\mapsto
\left<e^{it\Delta}f,g\right>, t\in \mathbb{R}^{1}$,
 is bounded on the real line and has an extension to the complex
plane as an entire function of the exponential type
$\omega^{2}+\|\rho\|^{2}$;

\item The abstract-valued
function $t\mapsto e^{it\Delta}f$  is bounded on the real line and has an
extension to the complex plane as an entire function of the
exponential type $\omega^{2}+\|\rho\|^{2}$;

\item  
A function $f\in L_{2}(X)$ belongs to the space $PW_{\omega_{f}}(X),
0<\omega_{f}<\infty,$ if and only if $f$ belongs to the set
$C^{\infty}(X)$, the limit
$$
 \lim_{k\rightarrow \infty}\|\Delta^{k}f\|^{1/k}
$$
exists and
\begin{equation}
\lim_{k\rightarrow
\infty}\|\Delta^{k}f\|^{1/k}=\omega_{f}^{2}+\|\rho\|^{2}.\label{limitcond}
\end{equation}

\item a function$f\in L_{2}(X)$ belongs to $PW_{\omega}(X)$ if and
only if $f\in C^{\infty}(X)$ and  the upper bound
\begin{equation}
\sup _{k\in \mathbb{N} }\left((\omega
^{2}+\|\rho\|^{2})^{-k}\|\Delta^{k}f\|\right)<\infty
\end{equation}
is finite,

\item a function $f\in L_{2}(X)$ belongs to $PW_{\omega}(X)$ if
and only if $f\in C^{\infty}(X)$ and
\begin{equation}
\underline{\lim}_{k\rightarrow\infty}\|\Delta^{k}f\|^{1/k}=\omega^{2}+\|\rho\|^{2}<\infty.
\end{equation}
In this case $\omega=\omega_{f}$.
\item The solution $u(t),
t\in \mathbb{R}^{1}$, of the Cauchy problem
$$
i\frac{\partial u(t)}{\partial t}=\Delta u(t), u(0)=f,
i=\sqrt{-1},
$$
has a holomorphic extension $u(z)$ to the
complex plane $\mathbb{C}$ satisfying
$$
\|u(z)\|_{L_{2}(X)}\leq e^{(\omega^{2}+\|\rho\|^{2})|\Im
z|}\|f\|_{L_{2}(X)}.
$$
\end{enumerate}
\label{PW}

\end{thm}

Now we are going to prove the following density  result which shows that for every $\omega>0$ the subspace  $PW_{\omega}(X)$  contains "many" functions.
\begin{thm}
For every $\omega>0$ and every open set $V\subset X$ if a function $f\in C_{0}^{\infty}(V)$ is orthogonal to all functions in  $PW_{\omega}(X)$ then $f$ is zero.
\end{thm}
\begin{proof}
 Assume that $f\in
C_{0}^{\infty}(V)$ is a such function and extend it by zero outside of $V$. By the Paley-Wiener Theorem and Parseval's formula 
the  transform $\widehat{f}(\lambda, b)$ is in $C^{\infty}(\textbf{a}^{*}_{+}\times
\mathcal{B})$ and  holomorphic in $\lambda$ and at the same time should be orthogonal
to all functions in $L_{2}\left(\Pi(0,\omega)\times B;
|c(\lambda)|^{-2}d\lambda db\right)$, where
$
\Pi(0,\omega)=\{\lambda \in \textbf{a}^{*}:
\|\lambda\|\leq \omega\}.
$
It implies that $\widehat{f}$ is zero.  The theorem is proved.
\end{proof}

\subsection{On decay of Paley-Wiener functions}

In this section we closely follow Andersen \cite{A}. 
Let us introduce the following spherical function 
$$
\varphi_{0}(g)=\int_{K}e^{\rho A(k^{-1}g)}dk=\int_{K}e^{\rho H(g^{-1}k)}dk.
$$

\begin{defn}\label{Ss}
The $L_{2}$-Schwartz space $S^{2}(X)$ is introduced as the space of all $f\in C^{\infty}(X)$ such that
$$
\sup_{x\in X}(1+|x|)^{N}\varphi_{0}(x)^{-1}\left|Df(x)\right|\leq \infty,\>\>N\in \mathbb{N}\cup{0},
$$
for all $D\in U(\textbf{g})$, where $U(\textbf{g})$ is the universal enveloping algebra of $\textbf{g}$. Here $|x|=|g|$, for $x=gK\in X$ where $|\cdot|$ is defined in (\ref{X-norm}).
\end{defn}

The space $S^{2}(X)$ can also be characterized as the space of all functions for which 
$$
(1+|g|)^{N}Df(g)\in L_{2}(X),
\>\>N\in \mathbb{N}\cup{0},
$$
for all $D\in U(\textbf{g})$, where $|x|=|g|$, for $x=gK\in X$.

Let $C_{0}^{\infty}(\textbf{a}^{*}\times \mathcal{B})^{W}$ be a subspace of functions in $C_{0}^{\infty}(\textbf{a}^{*}\times \mathcal{B})$ that satisfy the symmetry condition (\ref{symmetry}) for all $w\in W, \>\>\lambda\in \textbf{a}^{*}, \>\>x\in X$. 

The following fact is proved by M. Eguchi in \cite{Egu}, Theorem 4.1.1.
\begin{thm}\label{E}
The inverse Helgason-Fourier transform (\ref{IHF}) maps  $C_{0}^{\infty}(\textbf{a}^{*}\times \mathcal{B})^{W}$ into Schwartz space $S^{2}(X)$.
\end{thm}

In order to give a more detailed statement we will introduce several notations.
\begin{defn}
The space $PWS_{\omega}(X)$ is defined as the set of all functions $f$ in $PW_{\omega}(X)$ such that for all natural $m, n$
$$
(1+|x|)^{m}\Delta^{n}f(x)\in L_{2}(X),
$$
where $|x|=|g|$, for $x=gK\in X$.
The space $PWS(X)$ is defined as the union $\bigcup_{\omega>0}PWS_{\omega}(X).$
\end{defn}
We will also need the function subspace 
$
C_{\omega}^{\infty}(\textbf{a}^{*}\times \mathcal{B})$ which is defined as the set of all functions $f$ in  $  C_{0}^{\infty}(\textbf{a}^{*}\times \mathcal{B}) $ for which 
$$
\sup_{(\lambda, b)\in supp \widehat{f}}\|\lambda\|=\omega.
$$
The space $
C_{\omega}^{\infty}(\textbf{a}^{*}\times \mathcal{B})^{W}$ is a subspace of functions in $
C_{\omega}^{\infty}(\textbf{a}^{*}\times \mathcal{B})$ that satisfy symmetry condition (\ref{symmetry})  for all $w\in W, \>\>\lambda\in \textbf{a}^{*}, \>\>x\in X$. 

The following theorem was proved by N. B. Andersen in \cite{A}, Theorem 5.7.
\begin{thm}\label{A}
The  inverse Helgason-Fourier transform is a bijection of  $C_{0}^{\infty}(\textbf{a}^{*}\times \mathcal{B})^{W}$ onto $PWS(X)$, mapping $
C_{\omega}^{\infty}(\textbf{a}^{*}\times \mathcal{B})^{W}$ onto $PWS_{\omega}(X)$.
\end{thm}

\section{Average sampling and almost Parseval frames in Paley-Wiener spaces on Riemannian manifolds}\label{AvS}

Let $X_{r}=\{x_{k}\}$  be a $r$-lattice and $\{B(x_{k},r)\}$ be an associated family of balls that satisfy only properties (1) and (2) of the Lemma \ref{cover}.
We define
$$
U_{1}=B(x_{1}, r/2)\setminus \cup_{i,\>i\neq 1}B(x_{i}, r/4),
$$
and
$$
U_{k}=B(x_{k},  r/2)\setminus \left(\cup_{j<k}U_{j}\cup_{i,\>i\neq k}B(x_{i},  r/4)\right).
$$

One can verify the following. 
\begin{lem} The sets $\left\{U_{k}\right\}$ form a disjoint measurable cover of $X$ and
\begin{equation}\label{disjcover}
B(x_{k}, r/4)\subset U_{k}\subset B(x_{k}, r/2)
\end{equation}
\end{lem}

With every $U_{k}$ we associate a function $\psi_{k}\in C_{0}^{\infty}(U_{k})$ and we will always assume that for every $k$ it is not identical zero and
\begin{equation}\label{assumption}
0\leq \psi_{k}\leq 1,\>\>\>
\end{equation}
We introduce the following family $\Psi=\left\{\Psi_{k}\right\}$ of functionals $\Psi_{k}$ on $L_{2}(X)$:
 $$
 \Psi_{k}(F)=\frac{1}{|U_{k}|_{\psi_{k}} }\int_{U_{k}}F(x)\psi_{k}(x)d\mu(x)=\frac{1}{|U_{k}|_{\psi_{k}} }\int_{exp^{-1}_{y_{\nu}}(U_{k})}F(exp_{y_{\nu}}(x))\psi_{k}(exp_{y_{\nu}}(x))dx,\>\>\>
 $$
 where $x=(x_{1},...x_{d}),\>dx=dx_{1}...dx_{d}$ and
 $$|U_{k}|_{\psi_{k}} =\int_{exp^{-1}_{y_{\nu}}(U_{k})}\psi_{k}(exp_{y_{\nu}}x)dx=\int_{U_{k}}\psi_{k}(x)d\mu(x).
 $$
 Our {\it local Poincare-type inequality} is the following  \cite{Pes04b}.
\begin{lem}\label{LPI}
For    $m>d/2$ there exist constants $C=C(X, m)>0, \>\>r(X,m)>0,$
 such that for any $r$-lattice $X_{r}$ with $r<r(X,m)$ and any associated functional $\Psi_{k}$  the following inequality holds true for $f\in H^{m}(X)$: 
\begin{equation}\label{GPII}
\|(\varphi_{\nu}f)-\Psi_{k}((\varphi_{\nu}f))\|^{2}_{L_{2}(U_{k})}\leq
C(X, m)\sum_{1\leq |\alpha|\leq
m}r^{2|\alpha|}\|\partial^{\alpha}(\varphi_{\nu}f)\|^{2}_{L_{2}(B(x_{k}, r))},
\end{equation}
where  for $\alpha=(\alpha_{1},...,\alpha_{d})$
 $\partial^{\alpha}f=\partial^{\alpha_{1}}...\partial^{\alpha_{d}}f$ is a partial
derivative of order $|\alpha|$ in a fixed geodesic  coordinate system $\exp^{-1}_{y_{\nu}}$ in $B(y_{\nu}, \lambda)$ (see (\ref{Sob})).
\end{lem}

\begin{proof}It is obvious that the following relations hold:
$$
B_{T_{y_{\nu}}}(x_{k}, r/4)=exp^{-1}_{y_{\nu}}B(x_{k}, r/4)\subset exp^{-1}_{y_{\nu}}U_{k}\subset exp^{-1}_{y_{\nu}}B(x_{k}, r)=B_{T_{y_{\nu}}}(x_{k}, r)
$$
For every smooth $f$   and all
 $y=(y_{1},...,y_{d})\in exp^{-1}_{y_{\nu}}U_{k}\subset B(x_{k}, r/2),\>\>x=(x_{1},...,x_{d})\in B(x_{k}, r/2)$, we have the
following
$$
\varphi_{\nu}f(\exp_{y_{\nu}} x)   =
\varphi_{\nu}f(\exp_{y_{\nu}} y)    +
\sum_{1\leq|\alpha|\leq m-1} \frac{1}{\alpha
!}\partial^{\alpha}\varphi_{\nu}f(\exp_{y_{\nu}} y)
(x-y)^{\alpha}+
$$
\begin{equation}\label{TF}
\sum_{|\alpha|=m}\frac{1}{\alpha !}\int_{0}^{\eta}t^{m-1}\partial
^{\alpha}\varphi_{\nu}f(\exp_{y_{\nu}}(y+t\vartheta))\vartheta^{\alpha}dt,
\end{equation}
where 
$ \>\>\>\alpha=(
\alpha_{1},...,\alpha_{d}),\>\>\>\alpha!=\alpha_{1}!...\alpha_{d}!$
, $(x-y)^{\alpha}=(x_{1}-y_{1}))^{\alpha_{1}}...
(x_{d}-y_{d})^{\alpha_{d}},\>\>$
$ \eta=\|x-y\|,\>\> \vartheta=(x-y)/\eta.$

We multiply this inequality by $\psi_{k}(exp_{y_{\nu}}y)$ and integrate over
 $exp^{-1}_{y_{\nu}}U_{k}\subset B_{T_{y_{\nu}}}(x_{k}, r)$ with respect  to $d\mu( y)$. It gives
$$
\varphi_{\nu}f(\exp_{y_{\nu}}x)-\Psi_{k}(\varphi_{\nu}f)=
$$
$$
|U_{k}|_{\psi_{k}}^{-1}\int_{exp^{-1}_{y_{\nu}}U_{k}}\left(\sum_{1\leq|\alpha|\leq
m-1} \frac{1}{\alpha !}\partial^{\alpha}\varphi_{\nu}f(\exp_{y_{\nu}}y)(x-y) ^{\alpha}\right)\psi_{k}(exp_{y_{\nu}}y)dy+
$$
$$|U_{k}|_{\psi_{k}}^{-1}\int_{exp^{-1}_{y_{\nu}}U_{k}}\left(
\sum_{|\alpha|=m}\frac{1}{(m-1)!}\int_{0}^{\eta}t^{m-1}\partial
^{\alpha}\varphi_{\nu}f(\exp_{y_{\nu}}y+t\vartheta)\vartheta^{\alpha}dt\right)\psi_{k}(exp_{y_{\nu}}y)dy,
$$
where 
$$
\Psi_{k}(\varphi_{\nu}f\circ\exp_{y_{\nu}})=\int_{exp^{-1}_{y_{\nu}}U_{k}}\varphi_{\nu}f\circ\exp_{y_{\nu}}(x)\psi_{k}(\exp_{y_{\nu}}(x))dx=
$$
$$
\int_{U_{k}}\varphi_{\nu}f(x)\psi_{k}(x)d\mu(x)=\Psi_{k}(\varphi_{\nu}f).
$$
Then since $\psi_{k}\geq 0$

$$
\left|\varphi_{\nu}f(\exp_{y_{\nu}}x)-\Psi_{k}(\varphi_{\nu}f)\right|\leq
$$
$$
|U_{k}|_{\psi_{k}}^{-1}\sum_{1\leq|\alpha|\leq
m-1} \frac{1}{\alpha !}\int_{exp^{-1}_{y_{\nu}}U_{k}}\left|\partial^{\alpha}\varphi_{\nu}f(\exp_{y_{\nu}}y)(x-y) ^{\alpha}\right|\psi_{k}(exp_{y_{\nu}}y)dy+
$$
$$|U_{k}|_{\psi_{k}}^{-1}\sum_{|\alpha|=m}\frac{1}{(m-1)!}
\int_{exp^{-1}_{y_{\nu}}U_{k}}\left|\int_{0}^{\eta}t^{m-1}\partial
^{\alpha}\varphi_{\nu}f(\exp_{y_{\nu}}y+t\vartheta)\vartheta^{\alpha}dt\right|\psi_{k}(exp_{y_{\nu}}y)dy,
$$
We square this inequality and integrate over $exp^{-1}_{y_{\nu}}U_{k}$:
$$\|\varphi_{\nu}f-\Psi_{k}(\varphi_{\nu}f)\|^{2}_{L_{2}(U_{k})}=
\|\varphi_{\nu}f(\exp_{y_{\nu}}x)-\Psi_{k}(\varphi_{\nu}f)\|^{2}_{L_{2}(exp^{-1}_{y_{\nu}}U_{k})}\leq
$$
$$
C(m)|U_{k}|_{\psi_{k}}^{-2}\sum_{1\leq|\alpha|\leq
m-1}\int_{exp^{-1}_{y_{\nu}}U_{k}}\left(
\int_{exp^{-1}_{y_{\nu}}U_{k}}|\partial^{\alpha}\varphi_{\nu}f(\exp_{y_{\nu}}y)(x-y) ^{\alpha}|\psi_{k}(exp_{y_{\nu}}y)dy\right)^{2}dx+
$$
\begin{equation}\label{Two-integrals}
C(m)|U_{k}|_{\psi_{k}}^{-2}
\sum_{|\alpha|=m}\int_{exp^{-1}_{y_{\nu}}U_{k}}\left(\int_{exp^{-1}_{y_{\nu}}U_{k}}\left |\int_{0}^{\eta}t^{m-1}\partial
^{\alpha}\varphi_{\nu}f(\exp_{y_{\nu}}y+t\vartheta)\vartheta^{\alpha}dt \right |\psi_{k}(exp_{y_{\nu}}y)dy\right)^{2}dx=
$$
$$
I+II.
\end{equation}
Since $exp^{-1}_{y_{\nu}}U_{k}\subset B_{T_{y_{\nu}}}(x_{k},r/2),\>\>x, y\in U, $ one has $x-y\in B(x_{k},r)$, and an application of the Schwartz inequality gives 

$$
\int_{exp^{-1}_{y_{\nu}}U_{k}}|\partial^{\alpha}\varphi_{\nu}f(\exp_{y_{\nu}}y)(x-y) ^{\alpha}|\psi_{k}(exp_{y_{\nu}}y)dy\leq
$$
$$
 r^{|\alpha|}|U_{k}|^{1/2}_{\psi_{k}}\|\partial^{\alpha}\varphi_{\nu}f\circ exp_{y_{\nu}}\|_{L_{2}(exp^{-1}_{y_{\nu}}B(x_{k}, r))}.
$$
After all we obtain
\begin{equation}\label{1term}
I\leq C(X,m)\sum_{1\leq 
|\alpha|\leq m-1}r^{2|\alpha|}\|\partial^{\alpha}\varphi_{\nu}f\|^{2}_{L_{2}(B(x_{k}, 2r))}.
\end{equation}
Another application of the Schwartz inequality gives for $|\alpha|=m$
$$\left(\int_{exp^{-1}_{y_{\nu}}U_{k}}\left |\int_{0}^{\eta}t^{m-1}\partial
^{\alpha}\varphi_{\nu}f(\exp_{y_{\nu}}y+t\vartheta)\vartheta^{\alpha}dt \right |\psi_{k}(exp_{y_{\nu}}y)dy\right)^{2}\leq
$$
$$
\left(\int_{exp^{-1}_{y_{\nu}}U_{k}}\psi_{k}(exp_{y_{\nu}}y)dy\right)\left(\int_{exp^{-1}_{y_{\nu}}U_{k}}\left |\int_{0}^{\eta}t^{m-1}\partial
^{\alpha}\varphi_{\nu}f(\exp_{y_{\nu}}y+t\vartheta)\vartheta^{\alpha}dt \right |^{2}\psi_{k}(exp_{y_{\nu}}y)dy\right)=
$$
$$
|U_{k}|_{\psi_{k}}\left(\int_{exp^{-1}_{y_{\nu}}U_{k}}\left |\int_{0}^{\eta}t^{m-1}\partial
^{\alpha}\varphi_{\nu}f(\exp_{y_{\nu}}y+t\vartheta)\vartheta^{\alpha}dt \right |^{2}\psi_{k}(exp_{y_{\nu}}y)dy\right).
$$
By the  same Schwartz inequality using the  assumption $m>d/2$ one can  obtain the following estimate 
$$
\left |\int_{0}^{\eta}t^{m-1}\partial
^{\alpha}\varphi_{\nu}f(exp_{y_{\nu}}y+t\vartheta)\vartheta^{\alpha}dt\right |^{2}\leq
C\eta^{2m-d}\int_{0}^{\eta}t^{d-1}|\partial^{\alpha}
\varphi_{\nu}f(exp_{y_{\nu}}y+t\vartheta)|^{2}dt.
$$
Next,
$$
\int_{exp^{-1}_{y_{\nu}}U_{k}}\left |\int_{0}^{\eta}t^{m-1}\partial
^{\alpha}\varphi_{\nu}f(exp_{y_{\nu}}y+t\vartheta)\vartheta^{\alpha}dt\right 
|^{2}\psi_{k}(exp_{y_{\nu}}y)dy\leq
$$
$$C\int_{exp^{-1}_{y_{\nu}}U_{k}}
\eta^{2m-d}\int_{0}^{\eta}t^{d-1}|\partial^{\alpha}\varphi_{\nu}f(exp_{y_{\nu}}y+t\vartheta)|^{2}dt\psi_{k}(exp_{y_{\nu}}y)dy
$$
Thus, we have
$$
II\leq
$$
$$
C(X,m)|U_{k}|_{\psi_{k}}^{-1}\sum_{|\alpha|=m}\int_{exp^{-1}_{y_{\nu}}U_{k}}\left(\int_{exp^{-1}_{y_{\nu}}U_{k}}
\eta^{2m-d}\int_{0}^{\eta}t^{d-1}|\partial^{\alpha}\varphi_{\nu}f(exp_{y_{\nu}}y+t\vartheta)|^{2}dtdx\right)\psi_{k}(exp_{y_{\nu}}y)dy.
$$
We perform integration in parentheses  using the spherical coordinate system
$(\eta,\vartheta).$ Since $\eta\leq r$ we obtain for $|\alpha|=m$ 
\begin{equation}\label{RT-2}
\int_{0}^{r/2}\eta^{d-1}\int_{0}^{2\pi}
\left |\int_{0}^{\eta}t^{m-1}\partial
^{\alpha}\varphi_{\nu}f(exp_{y_{\nu}}y+t\vartheta)\vartheta^{\alpha}dt\right |^{2} d\vartheta
d\eta\leq
$$
$$C\int_{0}^{r/2}t^{d-1}\left(\int_{0}^{2\pi}\int_{0}^{r}
\eta^{2m-d}|\partial^{\alpha}\varphi_{\nu}f(exp_{y_{\nu}}y+t\vartheta)|^{2}
\eta^{d-1}d\eta d\vartheta\right)dt\leq
$$
$$
C(X,m)r^{2m}\|\partial^{\alpha}
\varphi_{\nu}f\|^{2}_{L_{2}(B(x_{k},r))},\>\>\>|\alpha|=m.
\end{equation}
Now the estimate (\ref{Two-integrals})
along with  (\ref{1term}) and (\ref{RT-2}) imply (\ref{GPII}). Lemma is proved.
\end{proof}

We introduce the following set of functionals
$$
\mathcal{A}_{k}(f)=\sqrt{|U_{k}}|\Psi_{k}(f)=\frac{\sqrt{|U_{k}}|} {|U_{k}|_{\psi_{k}}    }\int_{U_{k}}f(x)\psi_{k}d\mu(x),
$$
where 
$$
|U_{k}|=\int_{U_{k}} d\mu(x),
$$
and
$$
|U_{k}|_{\psi_{k}}=\int_{U_{k}} \psi_{k}(x)d\mu(x),\>\>\>\psi_{k}\in C_{0}^{\infty}(U_{k}),
$$
where $\psi_{k}$  is not identical to zero. Our {\it global Poincare-type inequality} is the following (compare to \cite{Pes00}, \cite{Pes04b}).
\begin{lem}\label{GPI}
For any   $0<\delta<1$ and $m>d/2$ there exist constants $c=c(X),\>C=C(X,m),$  such that the  following inequality holds true for any $r$-lattice with $r<c\delta$
\begin{equation}\label{gpi}
(1-2\delta/3)\|f\|^{2} \leq
\sum_{k}|\mathcal{A}_{k}(f)|^{2}+C\delta^{-1}r^{2m}\|\Delta^{m/2}f\|^{2}
\, \, \mbox{for all } f\in H^{m}(X).
\end{equation}
\end{lem}

\begin{proof} 

We will need the  inequality (\ref{ineq-1}) below. 
One has  for all $\alpha>0$
$$
|A|^{2}=|A-B|^{2}+2|A-B||B|+|B|^{2},\>\>\>\>\>
2|A-B||B|\leq\alpha^{-1}|A-B|^{2}+\alpha|B|^{2},
$$
which imply  the inequality 
$$
(1+\alpha)^{-1}|A|^{2}\leq\alpha^{-1}|A-B|^{2}+|B|^{2},\>\>\alpha>0.
$$
If, in addition, $0<\alpha<1$, then    one has 
\begin{equation}\label{ineq-1}
(1-\alpha)|A|^{2}\leq \frac{1}{\alpha}|A-B|^{2}+|B|^{2},\>\>0<\alpha<1.
\end{equation}
We have 
$$
(1-\alpha)|\varphi_{\nu}f|^{2}\leq \frac{1}{\alpha}|\varphi_{\nu}f-\Psi_{k}(\varphi_{\nu}f)|^{2}+\left|\Psi_{k}(\varphi_{\nu}f)\right|^{2},
$$
and since $U_{k}$ form a disjoint cover of $X$  we obtain
\begin{equation}\label{ineq-2}
(1-\alpha)\|\varphi_{\nu}f\|^{2}_{L_{2}(B(y_{\nu},\lambda/2))}=\sum_{k}
(1-\alpha)\|\varphi_{\nu}f\|^{2}_{L_{2}(U_{k})}\leq 
$$
$$
\alpha^{-1}\sum_{k}\|\varphi_{\nu}f-\Psi_{k}(\varphi_{\nu}f)\|^{2} _{L_{2}(U_{k})}+
\sum_{k}|U_{k}||\Psi_{k}(\varphi_{\nu}f)|^{2},\>\>\>|U_{k}|=\int_{U_{k}}d\mu(x).
\end{equation}
For $\alpha=\delta/3$  by using (\ref{GPII}) we have
$$
(1-\delta/3) \|\varphi_{\nu}f\|^{2}_{L_{2}(B(y_{\nu},\lambda/2))}\leq \sum_{k}|U_{k}| |\Psi_{k}(\varphi_{\nu}f)|^{2}+
3\delta^{-1}
\sum_{k}\|\varphi_{\nu}f-\Psi_{k}(\varphi_{\nu}f)\|^{2} _{L_{2}(U_{k})}\leq
$$
$$
\sum_{k}|U_{k}| |\Psi_{k}(\varphi_{\nu}f)|^{2}+
3C\delta^{-1}\sum_{k}\sum_{1\leq |\alpha|\leq
m}r^{2|\alpha|}\|\partial^{\alpha}(\varphi_{\nu}f)\|^{2}_{L_{2}(B(x_{k}, r)}\leq
$$
$$
\sum_{k}|U_{k}| |\Psi_{k}(\varphi_{\nu}f)|^{2}+
3CN_{X}\delta^{-1}\sum_{1\leq |\alpha|\leq
m}r^{2|\alpha|}\|\partial^{\alpha}(\varphi_{\nu}f)\|^{2}_{L_{2}((B(y_{\nu},\lambda/2))},
$$
and summation over $\nu$ gives 
$$
(1-\delta/3) \|f\|^{2}_{L_{2}(X)}=(1-\delta/3)\sum_{\nu} \|\varphi_{\nu}f\|^{2}_{L_{2}(B(y_{\nu},\lambda/2))}\leq
$$
$$
\sum_{\nu}\sum_{k}|U_{k}| |\Psi_{k}(\varphi_{\nu}f)|^{2}+
3CN_{X}\delta^{-1}\sum_{\nu}\sum_{1\leq |\alpha|\leq
m}r^{2|\alpha|}\|\partial^{\alpha}(\varphi_{\nu}f)\|^{2}_{L_{2}((B(y_{\nu},\lambda/2)))}\leq 
$$
$$
\sum_{\nu}\sum_{k}|U_{k}| |\Psi_{k}(\varphi_{\nu}f)|^{2}+
3CN_{X}\delta^{-1}\sum_{j=1}^{m}r^{2j}\|f\|^{2}_{H^{j}(X)}
$$

The regularity theorem for the elliptic second-order differential   operator $\Delta$   shows that for all $j\leq m$ there exists a $b=b(X,m)$ such that 
 \begin{equation}\label{reg}
\|f\|^{2}_{H^{j}(X)}\leq b\left(\|f\|^{2}_{L_{2}(X)}+\|\Delta^{j/2}f\|^{2}_{L_{2}(X)}\right),\>\>f\in \mathcal{D}(\Delta^{m/2}),\>\>b=b(X,m).
\end{equation}
Together with the following interpolation inequality which holds for general  self-adjoint operators 
\begin{equation}\label{interpol}
r^{2j}\|\Delta^{j/2}f\|^{2}_{L_{2}(X)}\leq 4a^{m-j}r^{2m}\|\Delta^{m/2}f\|^{2}_{L_{2}(X)}
+ca^{-j}\|f\|^{2}_{L_{2}(X)}, \>\>\>c=c(X, m),
\end{equation}
 for any $a,\>r>0, \>0\leq j\leq m$, it implies  that  there exists a constant  $C^{''}=C^{''}(X,m)$ such that the next inequality holds true
$$
(1-\delta/3)\|f\|^{2}_{L_{2}(X)}\leq 
\sum_{\nu}\sum_{k}|U_{k}| |\Psi_{k}(\varphi_{\nu}f)|^{2}+
$$
$$
C^{''}\left(r^{2}\delta^{-1}\|f\|^{2}_{L_{2}(X)}+r^{2m}\delta^{-1}\|\Delta^{m/2}f\|^{2}_{L_{2}(X)}+a^{-1}\|f\|^{2}_{L_{2}(X)}\right),
$$
where $ m>d/2.$  By choosing $a=(6C^{''}/\delta)>1$ we obtain, that there exists   a  constant  $C^{'''}=C^{'''}(X, m)$ such that for any $0<\delta<1$ and $\>\>\>r>0$
$$
(1-\delta/2)\|f\|^{2}_{L_{2}(X)}\leq 
\sum_{\nu}\sum_{k}|U_{k}| |\Psi_{k}(\varphi_{\nu}f)|^{2}+C^{'''}\left(r^{2}\delta^{-1}\|f\|^{2}_{L_{2}(X)}+
r^{2m}\delta^{-1}\|\Delta^{m/2}f\|^{2}_{L_{2}(X)}\right).
$$
 The last inequality shows, that if for  a given $0<\delta<1$ the value of $r$ is choosen such that 
 $$
 r<c\delta,\>\>\> \>\>c=\frac{1}{\sqrt{6C^{'''}}},\>\>\>\> C^{'''}=C^{'''}(X, m), 
 $$
then we obtain for a $m>d/2$
$$
(1-2\delta/3)\|f\|^{2}_{L_{2}(X)}\leq \sum_{\nu}\sum_{k}|U_{k}| |\Psi_{k}(\varphi_{\nu}f)|^{2}+C^{'''}\delta^{-1}r^{2m}\|\Delta^{m/2}f\|^{2}_{L_{2}(X)},
$$
where $|U_{k}|=\int_{U_{k}}d\mu(x).$
Lemma is proved.
\end{proof}
\begin{defn}\label{FrameDef}
We introduce  the following functions $\theta_{\nu, k}\in C_{0}^{\infty}(X)$
\begin{equation}\label{theta}
\theta_{\nu, k}=\frac{\sqrt{|U_{k}|}}{|U_{k}|_{\psi_{k}} }\psi_{k}\varphi_{\nu}.
\end{equation}
Thus  for any  $f\in L_{2}(X)$ we have
$$
\left<f, \theta_{\nu, k}\right>=\frac{\sqrt{|U_{k}|}}{|U_{k}|_{\psi_{k}} }\int_{U_{k}}f\psi_{k}\varphi_{\nu}d\mu(x). 
$$
\end{defn}

\begin{thm}\label{Frame-th}
There exists a  constant $a_{0}=a_{0}(X)$ such that, if for a given $0<\delta<1$ and an $\omega>0$ one has
\begin{equation}\label{rate}
r<a_{0}\delta^{1/d}(\omega^{2}+\|\rho\|^{2}) ^{-1/2},
\end{equation}
and the weight functions $\psi_{k}$ chosen in a way that 
\begin{equation}\label{measure-condition}
1\leq\sup_{k}\frac{|U_{k}|}{|U_{k}|_{\psi_{k}} }\leq 1+\delta,
\end{equation}
then for the  corresponding set of functions $\theta_{\nu, k}$  defined in (\ref{theta}) the following inequalities hold\begin{equation}\label{frame-ineq-PW}
(1-\delta)\|f\|^{2}_{L_{2}(X)}\leq \sum_{\nu}\sum_{k}\left |\left<f, \theta_{\nu, k}\right>\right|^{2}\leq(1+\delta)\|f\|^{2}_{L_{2}(X)},
$$
or
$$
(1-\delta)\|f\|^{2}_{L_{2}(X)}\leq \sum_{k}|U_{k}| |\Psi_{k}(\varphi_{\nu}f)|^{2}\leq (1+\delta)\|f\|^{2}_{L_{2}(X)},
\end{equation}
where $\>0<\delta<1,\>\>f\in PW_{\omega}(X).$
\end{thm}

\begin{rem}\label{rem}
Note that functions $\theta_{\nu, k}$ do not belong to $PW_{\omega}(X)$. The theorem actually says that projections $\theta_{ \nu; \omega, k}$ of these functions onto $PW_{\omega}(X)$ form an almost Parseval frame in $PW_{\omega}(X)$.
\end{rem}

\begin{proof}
By using  the Schwartz inequality 
we obtain for $f\in L_{2}(X)$
$$
\sum_{\nu}\sum_{k}|U_{k}| |\Psi_{k}(\varphi_{\nu}f)|^{2}=\sum_{\nu}\sum_{k}\frac{|U_{k}|}{|U_{k}|_{\psi_{k}} ^{2}}\left|\int_{U_{k}}\psi_{k}\varphi_{\nu}fd\mu(x)\right|^{2}\leq
$$
\begin{equation}\label{R-side}
\sum_{\nu}\sum_{k}\frac{|U_{k}|}{|U_{k}|_{\psi_{k}} }\int_{U_{k}}|\varphi_{\nu}f|^{2}d\mu(x)\leq (1+\delta)
\sum_{\nu} \int_{B(y_{\nu},\lambda)}|\varphi_{\nu}f|^{2}d\mu(x)=(1+\delta)\|f\|^{2}_{L_{2}(X)},
\end{equation}
where we used the assumption (\ref{measure-condition}).
According to the previous lemma, there exist $c=c(X),\>C=C(X)$ such that for any $0<\delta<1$ and any $\rho<c\delta$
\begin{equation}\label{ineq-3}
(1-2\delta/3)\|f\|^{2}_{L_{2}(X)}\leq 
\sum_{k}|U_{k}| |\Psi_{k}(\varphi_{\nu}f)|^{2}+Cr^{2d}\delta^{-1}\|\Delta^{d/2}f\|^{2}_{L_{2}(X)}.
\end{equation}
Notice, that if $f\in PW_{\omega}(X)$, then the Bernstein inequality holds
\begin{equation}\label{bern}
\|\Delta^{d/2}f\|^{2}_{L_{2}(X)}\leq (\omega^{2}+\|\rho\|^{2})^{d}\|f\|^{2}_{L_{2}(X)}.
\end{equation}
Inequalities (\ref{ineq-3}) and (\ref{bern})  show that for a certain $a_{0}=a_{0}(X)$, if  
$$
r<a_{0}\delta^{1/d}(\omega^{2}+\|\rho\|^{2})^{-1/2},
$$
 then 
\begin{equation}\label{L-side}
(1-\delta)\|f\|^{2}_{L_{2}(X)}\leq \sum_{k}|U_{k}| |\Psi_{k}(\varphi_{\nu}f)|^{2}\leq (1+\delta)\|f\|^{2}_{L_{2}(X)},
\end{equation}
where $\>0<\delta<1,\>\>f\in PW_{\omega}(X)$. Theorem is proved.

\end{proof}

\section{Nearly Parseval Paley-Wiener frames on  $X=G/K$}\label{FFFF}

\subsection{Paley-Wiener almost Parseval frames on $X=G/K$}\label{5.1}

 Let $g\in C^{\infty}(\mathbb{R}_{+})$ be a monotonic function such that $supp\>g\subset [0,\>  2], $ and $g(s)=1$ for $s\in [0,\>1], \>0\leq g(s)\leq 1, \>s>0.$ Setting  $Q(s)=g(s)-g(2s)$ implies that $0\leq Q(s)\leq 1, \>\>s\in supp\>Q\subset [2^{-1},\>2].$  Clearly, $supp\>Q(2^{-j}s)\subset [2^{j-1}, 2^{j+1}],\>j\geq 1.$ For the functions
 \begin{equation}\label{partition}
 F_{0}(\lambda, b)=\sqrt{ g(\|\lambda\|)\otimes 1_{\mathcal{B}}}, \>\>F_{j}(\lambda, b)=\sqrt{Q(2^{-j}\|\lambda\|)\otimes 1_{\mathcal{B}}},\>\>j\geq 1, \>\>\>
 \end{equation}
 one has 
 \begin{equation}\label{id-0}
 \sum_{j\geq 0}F_{j}^{2}(\lambda, b)=1_{\mathbb{R}^{n}\times{\mathcal{B}}}.
 \end{equation}
 By using the fact the Helgason-Fourier transform $\mathcal{F}$ and its inverse $\mathcal{F}^{-1}$ are isomorphisms between the spaces $L_{2}(X,d\mu(x))=L_{2}(X)$ and
$L_{2}(\textbf{a}^{*}_{+}\times \mathcal{B},
|c(\lambda)|^{-2}d\lambda db)$ and by using the Paseval's   formula  we can introduce the following self-adjoint bounded operator for every smooth compactly supported function $\Phi$ in $C_{0}^{\infty}(\mathbb{R}^{n}\times\mathcal{B})$ 
\begin{equation}\label{projector}
\Phi(\Delta)f=\mathcal{F}^{-1}\left(\Phi(\lambda, b)\mathcal{F}f(\lambda, b)\right)
\end{equation}
which maps  $L_{2}(X)$ onto $PW_{[\omega_{1},\>\>\omega_{2}]}(X)$ if $supp \>\Phi\subset [\omega_{1},\>\>\omega_{2}]$.  

We are using this definition in the case $F_{j}^{2}=\Phi$:
$$
F_{j}^{2}(\Delta)f=\mathcal{F}^{-1}\left(F_{j}^{2}(\lambda, b)\mathcal{F}f(\lambda, b)\right).
$$
Taking inner product with $f$  we obtain
$$
\|F_{j}(\Delta)f\|^{2}=\left<F_{j}^{2}(\Delta)f, f\right>
$$
and then (\ref{id-0}) gives
\begin{equation}
\label{norm equality-0}
\|f\|^2=\sum_{j\geq 0}\left< F_{j}^2(\Delta)f,f\right>=\sum_{j\geq 0}\|F_{j}(\Delta)f\|^2 .
\end{equation}

Let $\theta_{\nu, k}$ be the same as in Definition \ref{FrameDef} and $\theta_{\nu; j, k}\in PW_{2^{j+1}}(X)$ be their orthogonal projections on $ PW_{2^{j+1}}(X)$.

 According to    Theorem \ref{Frame-th} for a fixed $0<\delta<1$ there exists a constant $a_{0}=a_{0}(X)$ such that if for $\omega_{j}=2^{j+1}$ and 
$$
r_{j}=
a_{0}\delta^{1/d}(\omega_{j}^{2}+\|\rho\|^{2}          )^{-1/2}=a_{0}\delta^{1/d}(2^{2j+2}+\|\rho\|^{2})^{-1/2},\>\>\>j\in \mathbb{N}\cup {0},
$$
and the weight functions $\psi_{j,k}$ chosen in a way that 
\begin{equation}\label{measure-condition}
1\leq\sup_{j,k}\frac{|U_{j,k}|}{|U_{j,k}|_{\psi_{j,k}} }\leq 1+\delta,
\end{equation}
then the  set of functions $\theta_{\nu; j, k}\in PW_{2^{j+1}}(X)$ 
 form a frame in $PW_{2^{j+1}}(X)$ and
\begin{equation}\label{frame-ineq-PW}
(1-\delta)\|f\|^{2}\leq \sum_{\nu}\sum_{k}\left |\left<f, \theta_{\nu; j, k}\right>\right|^{2}\leq(1+\delta)\|f\|^{2},
\end{equation}
where $\>0<\delta<1,\>\>f\in PW_{2^{j+1}}(X).$ Since $F_{j}( \Delta)f\in PW_{2^{j+1}}(X)$ we can apply (\ref{projector}), (\ref{norm equality-0}), (\ref{frame-ineq-PW}) to obtain
\begin{equation}\label{Basic}
(1-\delta)\left \|F_{j}( \Delta)f\right\|^{2}\leq
\sum_{\nu}\sum _{ k}\left|\left<F_{j}( \Delta)f, \theta_{\nu; j,k}\right>\right|^{2}\leq(1+\delta) \left\| \>F_{j}( \Delta)f\right\|^{2}.
\end{equation}
Since operator $F_{j}(\Delta)$ is self-adjoint
we obtain (via (\ref{norm equality-0})), that for the functions
\begin{equation}\label{frame-functions}
\Theta_{\nu; j,k}=F_{j}( \Delta)\theta_{\nu; j,k}
\end{equation}
which belong  to $ PW_{[2^{j-1},\>\>2^{j+1}]}(X)$,
the following frame inequalities hold
\begin{equation}\label{frame-ineq}
(1-\delta)\|f\|^{2}\leq\sum_{j\geq 0}\sum_{\nu} \sum _{k}\left|\left<f, \Theta_{\nu; j,k}\right>\right|^{2}\leq(1+\delta) \|f\|^{2},\>\>\>\>f\in L_{2}(X).
\end{equation}

\subsection{Space localization of functions $\Theta_{\nu; j,k}$}

One has

\begin{equation}\label{frame-functions-1}
\Theta_{\nu; j,k}=F_{j}( \Delta)\theta_{\nu; j,k}=\frac{\sqrt{|U_{k}|}}{|U_{k}|_{\psi_{k}} }\mathcal{F}^{-1}\left(F_{j}(\lambda, b)\mathcal{F}  \psi_{j, k}\varphi_{\nu}(\lambda,b)\right)
\end{equation}
Since $\psi_{j, k}$ and $\varphi_{\nu}$ belong to $C_{0}^{\infty}(X)$ the Paley-Wiener Theorem for the Helgason-Fourier transform shows that $\mathcal{F} \psi_{j, k}\varphi_{\nu}$ belongs to $
\mathcal{H}(\textbf{a}^{*}_{\mathbb{C}}\times \mathcal{B})^{W}$. Note that $F_{j}$ belongs to $C_{0}^{\infty}(\mathbb{R}^{n}\times \mathcal{B})$, radial in $\lambda$, and independent on $b\in \mathcal{B}$. In other words the function $F_{j}(\lambda, b)\mathcal{F}  \psi_{j, k}\varphi_{\nu}$ belongs to $C_{0}^{\infty}(\textbf{a}^{*}\times \mathcal{B})^{W}$ and by the Eguchi Theorem \ref{E} the function $\mathcal{F}^{-1}\left(F_{j}(\lambda, b)\mathcal{F}  \psi_{j, k}\varphi_{\nu}\right)$ belongs to the Schwartz space $S^{2}(X)$.

To summarize, we proved the frame inequalities (\ref{frame-ineq}) for any $0<\delta<1$, where every function $\Theta_{\nu; j,k}$ belongs to 
$$
 PW_{[2^{2j-2},\>\>2^{2j+2}]}(X)\cap S^{2}(X).
$$
The Theorem \ref{FrTh} is proved.

\section{Besov spaces}\label{Besov}

We are going to remind a few basic facts from the theory of interpolation and  approximation spaces spaces \cite{BL}, \cite{BB}, \cite{KPS}.

Let $E$ be a linear space. A quasi-norm $\|\cdot\|_{E}$ on $E$ is
a real-valued function on $E$ such that for any $f,f_{1}, f_{2}\in
E$ the following holds true

\begin{enumerate}
\item $\|f\|_{E}\geq 0;$

\item $\|f\|_{E}=0  \Longleftrightarrow   f=0;$

\item $\|-f\|_{E}=\|f\|_{E};$

\item $\|f_{1}+f_{2}\|_{E}\leq C_{E}(\|f_{1}\|_{E}+\|f_{2}\|_{E}),
C_{E}>1.$

\end{enumerate}

We say that two quasi-normed linear spaces $E$ and $F$ form a
pair, if they are linear subspaces of a linear space $\mathcal{A}$
and the conditions $\|f_{k}-g\|_{E}\rightarrow 0,$ and
$\|f_{k}-h\|_{F}\rightarrow 0, f_{k}, g, h \in \mathcal{A}, $
imply equality $g=h$. For a such pair $E,F$ one can construct a
new quasi-normed linear space $E\bigcap F$ with quasi-norm
$$
\|f\|_{E\bigcap F}=\max\left(\|f\|_{E},\|f\|_{F}\right)
$$
and another one $E+F$ with the quasi-norm
$$
\|f\|_{E+ F}=\inf_{f=f_{0}+f_{1},f_{0}\in E, f_{1}\in
F}\left(\|f_{0}\|_{E}+\|f_{1}\|_{F}\right).
$$

All quasi-normed spaces $H$ for which $E\bigcap F\subset H \subset
E+F$ are called intermediate between $E$ and $F$. A group
homomorphism $T: E\rightarrow F$ is called bounded if
$$
\|T\|=\sup_{f\in E,f\neq 0}\|Tf\|_{F}/\|f\|_{E}<\infty.
$$
One says that an intermediate quasi-normed linear space $H$
interpolates between $E$ and $F$ if every bounded homomorphism $T:
E+F\rightarrow E+F$ which is a bounded homomorphism of $E$ into
$E$ and a bounded homomorphism of $F$ into $F$ is also a bounded
homomorphism of $H$ into $H$.

On $E+F$ one considers the so-called Peetere's $K$-functional
\begin{equation}
K(f, t)=K(f, t,E, F)=\inf_{f=f_{0}+f_{1},f_{0}\in E, f_{1}\in
F}\left(\|f_{0}\|_{E}+t\|f_{1}\|_{F}\right).\label{K}
\end{equation}
The quasi-normed linear space $(E,F)^{K}_{\theta,q}, 0<\theta<1,
0<q\leq \infty,$ or $0\leq\theta\leq 1,  q= \infty,$ is introduced
as a set of elements $f$ in $E+F$ for which
\begin{equation}
\|f\|_{\theta,q}=\left(\int_{0}^{\infty}
\left(t^{-\theta}K(f,t)\right)^{q}\frac{dt}{t}\right)^{1/q}.\label{Knorm}
\end{equation}

It turns out that $(E,F)^{K}_{\theta,q}, 0<\theta<1, 0\leq q\leq
\infty,$ or $0\leq\theta\leq 1,  q= \infty,$ with the quasi-norm
(\ref{Knorm})  interpolates between $E$ and $F$. 

Let us introduce another functional on $E+F$, where $E$ and $F$
form a pair of quasi-normed linear spaces
$$
\mathcal{E}(f, t)=\mathcal{E}(f, t,E, F)=\inf_{g\in F,
\|g\|_{F}\leq t}\|f-g\|_{E}.
$$

\begin{defn}The approximation space $\mathcal{E}_{\alpha,q}(E, F),
0<\alpha<\infty, 0<q\leq \infty $ is a quasi-normed linear spaces
of all $f\in E+F$ with the following quasi-norm
\begin{equation}
\left(\int_{0}^{\infty}\left(t^{\alpha}\mathcal{E}(f,
t)\right)^{q}\frac{dt}{t}\right)^{1/q}.
\end{equation}
\end{defn}

\begin{thm}\label{equivalence}
 Suppose that $\mathcal{T}\subset F\subset E$ are quasi-normed
linear spaces and $E$ and $F$ are complete.

 If there
exist $C>0$ and $\beta >0$ such that for any $f\in F$ the
following Jackson-type inequality is verified
\begin{equation}
t^{\beta}\mathcal{E}(t,f,\mathcal{T},E)\leq C\|f\|_{F},\label{dir}
t>0,
\end{equation}
 then the following embedding holds true
\begin{equation}
(E,F)^{K}_{\theta,q}\subset \mathcal{E}_{\theta\beta,q}(E,
\mathcal{T}), \>0<\theta<1, \>0<q\leq \infty.
\end{equation}

If there exist $C>0$ and $\beta>0$ such that for any $f\in \mathcal{T}$ the following Bernstein-type inequality holds
\begin{equation}
\|f\|_{F}\leq C\|f\|^{\beta}_{\mathcal{T}}\|f\|_{E}
\end{equation}
then 
\begin{equation}
\mathcal{E}_{\theta\beta, q}(E, \mathcal{T})\subset (F, F)^{K}_{\theta, q}, , \>0<\theta<1, \>0<q\leq \infty.
\end{equation}
\end{thm}\label{intthm}

Now we return to the situation on $X$. 
The inhomogeneous  Besov space $ \mathbf{B}_{2,q}^\alpha(X)$ is introduced as an interpolation space between the Hilbert space $L_{2}(X)$ and Sobolev space $H^{r}(X)$ where
 $r$ can be any natural number such that $0<\alpha<r, 1\leq
q\leq\infty$. Namely, we have 
$$
\mathbf{B}^{\alpha}_{2,q}(X)=( L_{2}(X), H^{r}(X))^{K}_{\theta, q},\>\>\> 0<\theta=\alpha/r<1,\>\>\>
1\leq q\leq \infty.
$$
where $K$ is the Peetre's interpolation functor.

We introduce a notion of best
approximation
\begin{equation}
\mathcal{E}(f,\omega)=\inf_{g\in
PW_{\omega}(X)}\|f-g\|_{L_{2}(X)}.\label{BA1}
\end{equation}

Our goal is to apply Theorem \ref{equivalence} in the situation where $E$ is the linear space $L_{2}(X)$ with its regular norm,  $\>F$ is the Sobolev space $H^{r}(X),$ with the graph norm $(I+\Delta)^{r/2}$, and $\mathcal{T}=PW_{\omega}(X)$ 
which is equipped with the quasi-norm
$$
\|f\|_{\mathcal{T}}=\sup\>\{\>|\lambda| : (\lambda, b)\in supp\>\widehat{f}\>\}.
$$

By using Plancherel Theorem it is easy to verify a  generalization of the Bernstein inequality for  Paley-Wiener functions $f\in PW_{\omega}(X)$:
$$
\|\Delta^{r}f\|_{L_{2}(X)}\leq \omega ^{r}\|f\|_{L_{2}(X)},\>\>\>r\in \mathbb{R}_{+},
$$
and  an analog  of the Jackson inequality:
$$
\mathcal{E}(f,\omega)\leq \omega^{-r}\|\Delta^{r}f\|_{L_{2}(X)},\>\>\>r\in \mathbb{R}_{+}.
$$

These two inequalities and Theorem \ref{equivalence}  imply the following result (compare to  \cite{Pes88}, \cite{Pes09a}- \cite{Pes11}).

\begin{thm} \label{approx}
The norm of the Besov space $\mathbf{B}_{2,q}^{\alpha}(X), \alpha>0, 1\leq q\leq\infty$ is
equivalent to the following norm
\begin{equation}
\|f\|_{L_{2}(X)}+\left(\sum_{k=0}^{\infty}\left(2^{k\alpha }\mathcal{E}(f,
2^{k})\right)^{q}\right)^{1/q}.
\end{equation}
\label{maintheorem1}
\end{thm}
Let	function $F_{j}$ be the same as in subsection \ref{5.1} and  
\begin{equation}\label{range}
F_j\left(\Delta\right):  L_{2}(X)  \rightarrow PW_{[2^{j-1},  2^{j+1}]}(X),\>\>\> \left\|F_j\left(\Delta\right)\right\|\leq1.
 \end{equation}

    \begin{thm}\label{mainLemma1}
     The norm of the Besov space $\mathbf{B}_{2,q}^{\alpha}(X)$  for $\alpha>0, 1\leq
q\leq \infty$ is equivalent to
\begin{equation} 
\left(\sum_{j=0}^{\infty}\left(2^{j\alpha
}\left \|F_j\left(\Delta\right)f\right \|_{L_{2}(X)}\right)^{q}\right)^{1/q},
\label{normequiv-1}
\end{equation}
  with the standard modifications for $q=\infty$.
\end{thm}

\begin{proof}

In the same notations as above the following version of Calder\'on decomposition holds:
$$
\sum_{j\in \mathbb{N}} F_j\left(\Delta\right)f= f,\> \>\>f\in L_{2}(X).
$$
We obviously have
$$
\mathcal{E}(f, 2^{k})\leq \sum_{j> k} \left \|F_j\left(\Delta\right)f\right \|_{L_{2}(X)}.
$$
By using the discrete Hardy inequality \cite{BB}  we obtain the estimate
\begin{equation} \label{direct}
\|f\|+\left(\sum_{k=0}^{\infty}\left(2^{k\alpha }\mathcal{E}(f,
2^{k})\right)^{q}\right)^{1/q}\leq C \left(\sum_{j=0}^{\infty}\left(2^{j\alpha
}\left \|F_j\left(\Delta\right)f\right \|_{L_{2}(X)}\right)^{q}\right)^{1/q}
\end{equation}
Conversely, 
 for any $g\in PW_{2^{j-1}}(X)$ we have
$$
\left\|F_j\left(\Delta\right)f\right\|_{L_{2}(X)}=\left\|F_{j}\left(\Delta\right)(f-g)\right\|_{L_{2}(X)}\leq \|f-g\|_{L_{2}(X)}.
$$
It gives the inequality 
$$
\left\|F_j\left(\Delta\right)f\right\|_{L_{2}(X)}\leq \mathcal{E}(f,\>2^{j-1}),
$$
which shows that the inequality opposite to (\ref{direct}) holds.
 This completes the proof.

\end{proof}

 \begin{thm}\label{mainLemma2}
     The norm of the Besov space $\mathbf{B}_{2,q}^{\alpha}(X)$  for $\alpha>0, 1\leq
q\leq \infty$ is equivalent  to
\begin{equation} 
 \left(\sum_{j=0}^{\infty}2^{j\alpha q }
\left(\sum_{\nu, k}\left|\left<f,\Theta_{\nu, j, k}\right>\right|^{2}\right)^{q/2}\right)^{1/q},
\label{normequiv}
\end{equation}
  with the standard modifications for $q=\infty$.
\end{thm}
\begin{proof} 

According to (\ref{Basic}) we have \begin{equation}
(1-\delta)\left \|F_j\left(\Delta\right)f\right \|_{L_{2}(X)}^{2}\leq 
\sum _{\nu, k}\left|\left< F_{j}\left({\Delta}\right)f, \theta_{\nu,j,k}\right>\right|^{2}\leq
(1+\delta)\left \|F_j\left(\Delta\right)f\right \|_{L_{2}(X)}^{2},
\end{equation}
where  $F_j\left(\Delta\right)f\in  PW_{2^{j+1}}(X)$. Since $\Theta_{\nu, j, k}=F_{j}\left(\Delta\right)\theta_{\nu, j, k}$ we obtain for any $f\in L_{2}(X)$ 
$$
\frac{1}{1+\delta}\sum_{\nu, k}\left|\left<f,\Theta_{\nu, j, k}\right>\right|^{2}\leq \left \|F_j\left(\Delta\right)f\right \|_{L_{2}(X)}^{2}\leq \frac{1}{1-\delta}\sum_{\nu, k}\left|\left<f,\Theta_{\nu, j, k}\right>\right|^{2}.
$$
Theorem is proved. 
\end{proof}

\makeatletter
\renewcommand{\@biblabel}[1]{\hfill#1.}\makeatother


\begin{thebibliography}{12}



\bibitem {A}
N. ~Andersen, {\em Real Paley-Wiener theorem for the inverse
Fourier transform on a Riemannian symmetric space}, Pacific J.
Math., 213 (2004), 1-13.

\bibitem{BS}
M. ~Birman and M. ~Solomyak, {\em Spectral thory of selfadjoint
operators in Hilbert space}, D.Reidel Publishing Co., Dordrecht,
1987.

\bibitem{BL}
J. ~Bergh, J. ~Lofstrom, {\em Interpolation spaces},
Springer-Verlag, 1976.


\bibitem{BB}
P. ~Butzer, H. ~Berens, {\em Semi-Groups of operators and
approximation}, Springer, Berlin, 1967 .


\bibitem{CGSM}
M. Calixto, J. Guerrero, J. C. Sanchez-Monreal, {\em Sampling theorem and discrete Fourier transform on the hyperboloid},  J. Fourier Anal. Appl. 17 (2011), no. 2, 240-264.

\bibitem{CKP}
T. Coulhon, G. Kerkyacharian, P. Petrushev, {\em
    Heat kernel generated frames in the setting of Dirichlet spaces},  arMiv:1206.0463.


\bibitem{EEKK1}
M. ~Ebata, M. ~Eguchi, S. ~Koizumi, K. ~Kumahara, {\em On sampling
formulas on symmetric spaces},  J. Fourier Anal. Appl. 12 (2006),
no. 1, 1--15.



\bibitem{EEKK2}
M. ~Ebata, M. ~Eguchi, S. ~Koizumi, K. ~Kumahara, {\em Analogues
of sampling theorems for some homogeneous spaces},  Hiroshima
Math. J. 36 (2006), no. 1, 125--140.


\bibitem{Egu}
M. Eguchi, {\em Asymptotic expansions of Eisenstein integrals and Fourier transform on
symmetric spaces}, J. Funct. Anal., 34 (1979), 167-216.

\bibitem{FP04}
H.~Feichtinger, I.~Pesenson,
{\em Iterative recovery of band limited functions on manifolds},
Contemp. Math., 2004, 137-153.

\bibitem{FP05}
H.~Feichtinger, I.~Pesenson,
{\em A reconstruction method for band-limited signals on the hyperbolic plane. Sampl. Theory Signal Image Process.}, 4 (2005), no. 2, 107-119.

\bibitem{gm3} D. Geller and A. Mayeli, {\em Nearly Tight Frames and Space-Frequency Analysis on Compact Manifolds} (2009),
Math. Z. {\bf 263} (2009), 235-264.

\bibitem{gm4} D. Geller and A. Mayeli, {\em Besov spaces and frames on compact manifolds}, Indiana Univ. Math. J. 58 (2009), no. 5, 2003-2042.


\bibitem{GM100} D. Geller and D. Marinucci, {\em Mixed needlets}, J. Math. Anal. Appl. 375 (2011), no. 2, 610-630.

\bibitem{gpes} D. Geller and I. Pesenson, {\em Band-limited localized Parseval frames and Besov spaces on compact homogeneous manifolds}, J. Geom. Anal. 21 (2011), no. 2, 334-37.


\bibitem {H}
E.~Hebey, {\em Sobolev Spaces on Riemannian Manifolds},
Springer-Verlag, Berlin, Heidelberg, 1996.



\bibitem {H2}
S. Helgason, {\em Groups and Geometric Analysis}, Pure and Applied Mathematics, 113. Academic Press, Inc., Orlando, FL, 1984. xix+654 pp. ISBN: 0-12-338301-3.

\bibitem{KPS}
S. ~ Krein, Y. ~Petunin, E. ~Semenov, {\em  Interpolation of
linear operators}, Translations of Mathematical Monographs, 54.
AMS, Providence, R.I., 1982.


\bibitem{NPW} F.J. Narcowich, P. Petrushev and J. Ward, {\em Localized Tight frames on spheres},
SIAM J. Math. Anal. 38, (2006), 574-594.

\bibitem{Pa}
A.~Pasquale, {\em A Paley-Wiener theorem for the inverse spherical
transform}, Pacific J. Math., 193 (2000), 143-176.


\bibitem{Pes88}
I.~Pesenson, {\em The Best Approximation in a Representation
Space of a Lie Group}, Dokl. Acad. Nauk USSR, v. 302, No 5, pp.
1055-1059, (1988) (Engl. Transl. in Soviet Math. Dokl., v.38, No
2, pp. 384-388, 1989.)

\bibitem{Pes98a}
 I. Pesenson, {\em Reconstruction of Paley-Wiener functions on the Heisenberg group},  Voronezh Winter Mathematical Schools, 207-216, Amer. Math. Soc. Transl. Ser. 2, 184, Amer. Math. Soc., Providence, RI, 1998. 

\bibitem{Pes98b}
I. Pesenson, {\em Sampling of Paley-Wiener functions on stratified
groups}, J. of Fourier Analysis and Applications {\bf 4} (1998),
269--280.

\bibitem{Pes00}
I. Pesenson, {\em A sampling theorem on homogeneous manifolds},
Trans. Amer. Math. Soc. {\bf 352} (2000), no. 9, 4257--4269.


\bibitem{Pes04b}
I.~Pesenson,  {\em Poincare-type inequalities and reconstruction
of Paley-Wiener functions on manifolds, }  J. of Geometric Analysis
, (4), 1, (2004), 101-121.


\bibitem{Pes06}
I. Pesenson, {\em Deconvolution of band limited functions on
symmetric spaces},  Houston J. of Math., 32, No. 1, (2006),
183-204.



\bibitem{Pes07}
I. ~Pesenson, {\em Frames in Paley-Wiener spaces on Riemannian
manifolds}, in Integral Geometry and Tomography, Contemp. Math.,
405, AMS, (2006), 137-153.

\bibitem{Pes09a}
I. ~Pesenson, {\em A Discrete Helgason-Fourier Transform for
Sobolev and Besov functions on noncompact symmetric spaces},
Contemp. Math, 464, AMS, (2008), 231-249.


\bibitem{Pes09}
 I. ~Pesenson, {\em Paley-Wiener approximations and multiscale approximations in
Sobolev and Besov spaces on manifolds,}  J. of Geometric Analysis, 4, (1), (2009), 101-121.


\bibitem{Pes11}
I. Pesenson, M. Pesenson, {\em Approximation of Besov vectors by Paley-Wiener vectors in Hilbert spaces},
  Approximation Theory MIII: San Antonio 2010 (Springer Proceedings in Mathematics), by Marian Neamtu and Larry Schumaker, 249-263.



\bibitem{Pes12}
 I. ~Pesenson, {\em Localized Bandlimited nearly tight frames and Besov spaces on domains in Euclidean spaces}, submitted,  	arMiv:1208.5165v1.



\bibitem {T1}
H. ~Triebel, {\em Spaces of Hardy-Sobolev-Besov type on complete
Riemannian manifolds}, Ark. Mat. 24, (1986), 299-337.


\bibitem {T2}
H. ~Triebel, {\em Function spaces on Lie groups}, J. London Math.
Soc. 35, (1987), 327-338.

\bibitem {T3}
H. ~Triebel, {\em  Theory of function spaces II,}
  Monographs in Mathematics, 84. Birkhauser Verlag, Basel, 1992.



\end{thebibliography}
\end{document}